\documentclass[11pt]{amsart}
\pdfoutput=1

\usepackage{amssymb,mathtools}
\usepackage{microtype}
\usepackage{mathtools}
\usepackage{tikz-cd}

\usepackage{forest}
\forestset{
  Bseries/.style={
    for tree={
      grow'=90,
      minimum width=3pt,
      inner sep=0pt,
      s sep=2.5pt,
      circle,
      draw,
      semithick,
      edge=semithick,
      fill,
      if level=0{
        baseline
      }{}
    },
    delay={
      where content={o}{content={}, fill=none}{}
    },
    before computing xy={
      for tree={
        l=5pt,
      }
    }
  },
}
\newcommand\Bseries[1]{\Forest{Bseries #1}}

\usepackage{enumitem}

\usepackage[margin=1in, marginparwidth=.75in]{geometry}

\usepackage[numbers,longnamesfirst]{natbib}

\usepackage{hyperref}
\hypersetup{
  pdftitle={Functional equivariance and modified vector fields},
  pdfauthor={Ari Stern and Sanah Suri},
  pdfsubject={MSC 2020: 37M15},
  bookmarksopen=false
}

\usepackage[capitalize,nameinlink]{cleveref}

\makeatletter
\g@addto@macro\bfseries{\boldmath}
\makeatother

\theoremstyle{plain} 
\newtheorem{theorem}{Theorem}[section]
\newtheorem{lemma}[theorem]{Lemma}
\newtheorem{corollary}[theorem]{Corollary}
\newtheorem{proposition}[theorem]{Proposition}

\theoremstyle{definition}
\newtheorem{definition}[theorem]{Definition}
\newtheorem{example}[theorem]{Example}
\newtheorem{assumption}[theorem]{Assumption}

\theoremstyle{remark}
\newtheorem{remark}[theorem]{Remark}

\begin{document}

\title{Functional equivariance and modified vector fields}

\author{Ari Stern}
\author{Sanah Suri}
\address{Department of Mathematics and Statistics, Washington University in St.~Louis}
\email{stern@wustl.edu}
\email{s.sanah@wustl.edu}

\begin{abstract}
  This paper examines functional equivariance, recently introduced by
  McLachlan and Stern [Found.\ Comput.\ Math.\ (2022)], from the
  perspective of backward error analysis. We characterize the
  evolution of certain classes of observables (especially affine and
  quadratic) by structure-preserving numerical integrators in terms of
  their modified vector fields. Several results on invariant
  preservation and symplecticity of modified vector fields are thereby
  generalized to describe the numerical evolution of non-invariant
  observables.
\end{abstract}

\maketitle

\section{Introduction}

\emph{Functional equivariance}, recently introduced by
\citet{McSt2022}, is a structure-preserving property of numerical
integrators describing the evolution of certain observables. Given
$ \dot{y} = f (y) $, the chain rule implies that a $ C ^1 $ observable
$ z = F (y) $ evolves according to $ \dot{z} = F ^\prime (y) f (y)
$. A numerical integrator is said to be \emph{$F$-functionally
  equivariant} if, when the integrator is applied to
$ \dot{y} = f (y) $ for any $f$, the numerical evolution of
$ \bigl( y , F (y) \bigr) $ is identical to that obtained by
numerically integrating the augmented system
\begin{equation*}
  \dot{y} = f (y) , \qquad \dot{z} = F ^\prime (y) f (y) .
\end{equation*}
If $\Phi$ is a one-step numerical integrator, $ \Phi _f $ is its
application to the vector field $f$, and $ \Phi _g $ is its
application to the augmented vector field
$ g( y, z ) = \bigl( f (y) , F ^\prime (y) f (y) \bigr) $, then this
is the statement that the following diagram commutes:
\begin{equation*}
  \begin{tikzcd}
    y _0 \ar[r, |->, "\Phi _f"] \ar[d, |->, "{(\mathrm{id}, F)}
    "'] & y _1 \ar[d, |->, "{(\mathrm{id}, F)}
    "]\\
    ( y _0, z _0 ) \ar[r, |->, "\Phi _g"] & ( y _1, z _1 ) \rlap{ .}
  \end{tikzcd}
\end{equation*}
A well-studied special case is when $ F ^\prime (y) f (y) = 0 $ and
$ \Phi $ preserves the invariance of $ F (y) $ whenever $F$ is, e.g.,
affine or quadratic; see \citet*[Chapter~IV]{HaLuWa2006} for a survey
of such results. However, there are many important cases discussed in
\citep{McSt2022} where one may wish to preserve the evolution of
\emph{non-invariant} observables, e.g., local conservation laws in
numerical PDEs.

In this paper, we examine functional equivariance from the perspective
of \emph{modified vector fields}, which form the foundation for
backward error analysis of numerical integrators;
cf.~\citep[Chapter~IX]{HaLuWa2006} and references therein. From this
point of view, a numerical trajectory of $ \dot{y} = f (y) $ is
formally viewed as a solution to a modified equation
$ \dot{ \widetilde{ y } } = \widetilde{ f } ( \widetilde{ y } ) $, so
the numerical evolution of the observable
$ \widetilde{ z } = F ( \widetilde{ y } ) $ is given by
$ \dot{ \widetilde{ z } } = F ^\prime ( \widetilde{ y } ) \widetilde{
  f } ( \widetilde{ y } ) $. Therefore, $F$-functional equivariance
corresponds to the condition
\begin{equation}
  \label{e:fe_mvf}
  \widetilde{ g } \bigl(  \widetilde{ y } , F ( \widetilde{y}) \bigr) = \bigl( \widetilde{ f } ( \widetilde{ y } ) , F ^\prime ( \widetilde{ y } ) \widetilde{ f } ( \widetilde{ y } ) \bigr) ,
\end{equation}
where $ \widetilde{g} $ is the modified vector field of the augmented
vector field $g$ given previously. As we will see, this approach
generalizes several well-known results relating invariant-preserving
integrators to modified vector fields, corresponding to the special
case where $ F ^\prime (y) f (y) = 0 $ implies
$ F ^\prime ( \widetilde{ y } ) \widetilde{ f } ( \widetilde{ y } ) =
0 $.

The paper is organized as follows:
\begin{itemize}
\item \Cref{s:integrator_maps} develops a theory of functional
  equivariance for \emph{integrator maps}
  $ \phi \colon f \mapsto \widetilde{ f } $, which take vector fields
  to modified vector fields, as in \citet{MuVe2016} and
  \citet{McMoMuVe2016}. This section proves integrator-map versions of
  the key results in \citep[Section~2]{McSt2022}.

\item \Cref{s:mvf} considers modified vector fields of one-step
  integrators $ \Phi $. In this setting, $ \widetilde{ f } $ is a
  formal power series in the step size, and its finite truncations are
  integrator maps in the sense of the previous section. For
  $F$-functionally equivariant integrators, in the sense of
  \citep{McSt2022}, we show that the condition \eqref{e:fe_mvf} holds
  term-by-term. Considering truncations therefore links the results in
  \citep{McSt2022} to those in \cref{s:integrator_maps} of the present
  paper.

\item Finally, \cref{s:additive_partitioned} discusses the
  generalization of the results in the preceding sections to additive
  and partitioned integrator maps and integrators, including
  additive/partitioned Runge--Kutta methods and splitting/composition
  methods.
  
\end{itemize}

\subsection*{Acknowledgments}

This material is based upon work supported by the National Science
Foundation under Grant No.~DMS-2208551.

\section{Functional equivariance of integrator maps}
\label{s:integrator_maps}

\subsection{Integrator maps and affine equivariance}

Let $ f \in \mathfrak{X} (Y) $ be a smooth vector field on a Banach
space $Y$, and denote its time-$h$ flow by
$ \exp h f \colon Y \rightarrow Y $. A one-step numerical integrator
$\Phi$ approximates this flow by
$ \Phi _{ h f } \colon Y \rightarrow Y $. In backward error analysis,
one views this as the flow of a modified vector field
$ \widetilde{ f } $, for which
$ \Phi _{ h f } = \exp h \widetilde{ f } $. However,
$ \widetilde{ f } $ is typically a formal power series in $h$, which
must be interpreted as a (possibly divergent) asymptotic expansion
rather than a genuine vector field (\citet*[Chapter~IX]{HaLuWa2006}).

To sidestep these technicalities, at least until \cref{s:mvf}, we
begin by following \citet{MuVe2016} and \citet*{McMoMuVe2016} in first
considering \emph{integrator maps}, whose modified vector fields are
genuine vector fields.

\begin{definition}
  An \emph{integrator map} $\phi$ is a collection of smooth maps
  $ \phi _Y \colon \mathfrak{X} (Y) \rightarrow \mathfrak{X} (Y) $,
  for each Banach space $Y$.  For $ f \in \mathfrak{X} (Y) $, we will
  typically write $ \phi (f) $ to mean the same thing as
  $ \phi _Y (f) $, which we call the \emph{modified vector field} of
  $f$ with respect to $\phi$. When $\phi$ is fixed, we will often
  denote the modified vector field simply by $ \widetilde{ f } $.
\end{definition}

\begin{remark}
  We allow for infinite-dimensional Banach spaces, which are used in
  some of the PDE applications discussed in \citet{McSt2022}. This is
  in contrast with \citep{MuVe2016,McMoMuVe2016}, who consider only
  vector fields on $ \mathbb{R}^n $ for $ n \in \mathbb{N} $. We are
  interested primarily in the \emph{algebraic} properties of these
  methods, and we do not attempt to address the tricky
  \emph{analytical} issues that may arise when considering backward
  error analysis of integrators on arbitrary Banach spaces.

  Fox example, any $\phi$ defined by a finite B-series (e.g., a finite
  truncation of the B-series for the modified vector field of a
  Runge--Kutta method) gives an integrator map in the sense of the
  definition above. Note that infinite B-series may diverge for
  certain $f$, even on $\mathbb{R}^n$.
\end{remark}

The integrator maps considered in this section will also be
\emph{affine equivariant} in the sense of \citep{McMoMuVe2016}. By the
main result of that paper, this means that they are B-series
maps. However, we will usually not use the equivalent characterization
of these maps in terms of trees and elementary differentials, instead
relying primarily on the affine equivariance property in the results
to follow.

\begin{definition}
  Given a G\^ateaux differentiable map
  $ \chi \colon Y \rightarrow U $, a pair of vector fields
  $ f \in \mathfrak{X} (Y) $ and $ g \in \mathfrak{X} (U) $ is
  \emph{$\chi$-related} if
  $ \chi ^\prime (y) f (y) = g \bigl( \chi (y) \bigr) $ for all
  $ y \in Y $, and we write $ f \sim _\chi g $. In particular, if
  $ A \colon Y \rightarrow U $ is affine, then $ f \sim _A g $
  whenever $ A ^\prime \circ f = g \circ A $. An integrator map $\phi$
  is \emph{affine equivariant} if $ f \sim _A g $ implies
  $ \phi (f) \sim _A \phi (g) $ for all affine maps $A$ between Banach
  spaces.
\end{definition}

\subsection{Functional equivariance}

We next define functional equivariance for integrator maps.

\begin{definition}
  \label{d:fe}
  Given a G\^ateaux differentiable map $ F \colon Y \rightarrow Z $
  and $ f \in \mathfrak{X} (Y) $, define the \emph{augmented vector
    field} $ g \in \mathfrak{X} ( Y \times Z ) $ by
  $ g ( y , z ) = \bigl( f (y) , F ^\prime (y) f (y) \bigr) $. An
  integrator map $ \phi $ is \emph{$F$-functionally equivariant} if
  $ \phi ( f ) \sim _{ ( \mathrm{id}, F ) } \phi (g) $ for all $f$,
  which is precisely the condition \eqref{e:fe_mvf}. Given a class of
  maps $\mathcal{F}$, we say that $\phi$ is
  \emph{$\mathcal{F}$-functionally equivariant} if it is
  $F$-functionally equivariant for all $ F \in \mathcal{F} ( Y, Z ) $
  and all Banach spaces $Y$ and $Z$.
\end{definition}

Let us now restrict our attention to affine equivariant integrator
maps. We first show that affine equivariance allows us to characterize
$ \widetilde{g} ( \widetilde{y} , \widetilde{z} ) $ for all
$ \widetilde{z} \in Z $, not just
$ \widetilde{z} = F ( \widetilde{y} ) $. This gives a stronger notion
of what it means for an affine equivariant integrator map to be
$F$-functionally equivariant: \emph{If $g$ is the augmented vector
  field of $f$, then $\widetilde{g}$ is the augmented vector field of
  $ \widetilde{f} $.}

\begin{proposition}
  \label{p:gtilde_affine}
  If $\phi$ is affine equivariant, then
  $ \widetilde{g} ( \widetilde{y} , \widetilde{z} ) = \widetilde{g}
  \bigl( \widetilde{y} , F ( \widetilde{y} ) \bigr) $ for all
  $ ( \widetilde{y}, \widetilde{z} ) \in Y \times Z $. Consequently,
  the $F$-functional equivariance condition \eqref{e:fe_mvf} holds if
  and only if
  $ \widetilde{g} ( \widetilde{y} , \widetilde{z} ) = \bigl(
  \widetilde{f} ( \widetilde{y} ) , F ^\prime ( \widetilde{y} )
  \widetilde{f} ( \widetilde{y} ) \bigr) $.
\end{proposition}

\begin{proof}
  Consider the affine map $ A ( y, z ) = ( y , z + c ) $, where
  $ c \in Z $ is a constant. Since $g$ depends only on $y$, we have
  $ g ( y, z ) = g ( y , z + c ) $, i.e., $ g \sim _A g $. Affine
  equivariance therefore implies
  $ \widetilde{g} \sim _A \widetilde{g} $, i.e.,
  $ \widetilde{g} ( \widetilde{y} , \widetilde{z} ) = \widetilde{g} (
  \widetilde{y} , \widetilde{z} + c ) $. For any fixed
  $ ( \widetilde{y} , \widetilde{z} ) \in Y \times Z $, taking
  $ c = F ( \widetilde{y} ) - \widetilde{z} $ completes the proof.
\end{proof}

We next consider the case where $ \mathcal{F} ( Y , Z ) $ is the class
of affine maps $ Y \rightarrow Z $. Compare the following result with
\citep[Proposition~2.6]{McSt2022}.

\begin{proposition}
  \label{p:ae_implies_afe}
  Every affine equivariant integrator map is affine functionally
  equivariant.
\end{proposition}

\begin{proof}
  If $ F \colon Y \rightarrow Z $ is affine, then so is
  $ ( \mathrm{id}, F ) \colon Y \rightarrow Y \times Z $. Since $f$
  and $g$ in \cref{d:fe} satisfy $ f \sim _{ (\mathrm{id}, F ) } g $,
  affine equivariance of $\phi$ implies
  $ \phi (f) \sim _{ (\mathrm{id}, F) } \phi (g) $.
\end{proof}

We now characterize functional equivariance with respect to more
general classes of maps $\mathcal{F}$, such as quadratic or
higher-degree polynomial maps. Since the integrator maps under
consideration are affine equivariant, we make the following natural set of assumptions on $\mathcal{F}$; cf.~\citep[Assumption~2.8]{McSt2022}.

\begin{assumption}
  \label{a:affine}
  The class of maps $\mathcal{F}$ satisfies the following:
  \begin{itemize}
  \item $ \mathcal{F} ( Y, Y) $ contains the identity map for all $Y$;
  \item $ \mathcal{F} ( Y, Z ) $ is a vector space for all $Y$ and $Z$;
  \item $ \mathcal{F} $ is invariant under composition with affine
    maps, in the following sense: If $ A \colon Y \rightarrow U $ and
    $ B \colon V \rightarrow Z $ are affine and
    $ F \in \mathcal{F} ( U, V ) $, then
    $ B \circ F \circ A \in \mathcal{F} ( Y, Z ) $.
  \end{itemize}
\end{assumption}

The main result of this section will relate functional equivariance to
the more well-studied notion of invariant preservation, which we now
recall.

\begin{definition}
  Given a G\^ateaux differentiable map $ F \colon Y \rightarrow Z $,
  an integrator map $\phi$ is \emph{$F$-invariant preserving} if
  $ F ^\prime f = 0 $ implies $ F ^\prime \phi (f) = 0 $ for all
  $f$. We say that $ \phi $ is \emph{$\mathcal{F}$-invariant
    preserving}, for a class of maps $\mathcal{F}$, if $\phi$ is
  $F$-invariant preserving for all $ F \in \mathcal{F} ( Y , Z ) $ and
  all Banach spaces $Y$ and $Z$.
\end{definition}

Although functional equivariance seems stronger than invariant
preservation, since it accounts for both invariant and non-invariant
observables $F$, the properties are in fact equivalent for affine
equivariant integrator maps. Compare the following result with
\citep[Theorem~2.9]{McSt2022}.

\begin{theorem}
  \label{t:fe_iff_invariant}
  Let $\mathcal{F}$ satisfy \cref{a:affine}. An affine equivariant
  integrator map $\phi$ is $\mathcal{F}$-invariant preserving if and
  only if it is $\mathcal{F}$-functionally equivariant.
\end{theorem}

\begin{proof}
  $ ( \Rightarrow ) $ Suppose $\phi$ is $\mathcal{F}$-invariant
  preserving. If $ F \in \mathcal{F} ( Y, Z ) $, \cref{a:affine}
  implies that $ G( y, z ) = F (y) - z $ is in
  $ \mathcal{F} ( Y \times Z , Z ) $. Furthermore, $G$ is an invariant
  of the augmented vector field $g$, since
  $ G ^\prime ( y, z ) g (y, z ) = F ^\prime (y) f (y) - F ^\prime (y)
  f (y) = 0 $. Writing
  $ \widetilde{ g } = ( \widetilde{g} _Y , \widetilde{g} _Z ) $, the
  fact that $\phi$ is $\mathcal{F}$-invariant preserving implies
  \begin{equation}
    \label{e:Gprime}
    0 = G ^\prime ( \widetilde{y} , \widetilde{z} ) \widetilde{g} ( \widetilde{y}, \widetilde{z}) = F ^\prime (\widetilde{y}) \widetilde{g} _Y ( \widetilde{y}, \widetilde{z}) - \widetilde{g} _Z ( \widetilde{y} , \widetilde{z}) ,
  \end{equation}
  Now, letting $ A ( y, z ) = y $ be linear projection onto the $Y$
  component, $ g \sim _A f $ implies
  $ \widetilde{ g } \sim _A \widetilde{ f } $ by affine
  equivariance. This says that
  $ \widetilde{ g } _Y ( \widetilde{y} , \widetilde{z} ) = \widetilde{
    f } (\widetilde{y}) $, so we conclude from \eqref{e:Gprime} that
  $ \widetilde{ g } _Z ( \widetilde{ y } , \widetilde{z} ) = F ^\prime
  ( \widetilde{y} ) \widetilde{f} ( \widetilde{y}) $. Hence, $\phi$ is
  $\mathcal{F}$-functionally equivariant.

  $ ( \Leftarrow ) $ Conversely, suppose $\phi$ is
  $\mathcal{F}$-functionally equivariant. If
  $ F \in \mathcal{F} ( Y , Z ) $ is an invariant of
  $f \in \mathfrak{X} (Y) $, then the augmented vector field is
  $ g ( y, z ) = \bigl( f (y) , 0 \bigr) $, and
  $\mathcal{F}$-functional equivariance implies
  $ \widetilde{g} \bigl( \widetilde{y}, \widetilde{z} \bigr) = \bigl(
  \widetilde{f} (\widetilde{y}), F ^\prime (\widetilde{y})
  \widetilde{f} (\widetilde{y}) \bigr) $ by
  \cref{p:gtilde_affine}. However, taking the linear projection
  $ B ( y, z ) = z $ gives $ g \sim _B 0 $, so affine equivariance
  implies $ \widetilde{g} \sim _B \widetilde{0} = 0 $. (As in
  \citep[Lemma~6.1]{McMoMuVe2016}, one proves $ \widetilde{ 0 } = 0 $
  by considering the affine map from the trivial Banach space to any
  point of $Z$.) Thus,
  $ F ^\prime (\widetilde{y}) \widetilde{f} (\widetilde{y}) = 0 $, so
  $\phi$ is $\mathcal{F}$-invariant preserving.
\end{proof}

\begin{example}
  \label{e:b-series_fe}
  We illustrate functional equivariance for some simple B-series
  integrator maps, whose terms are elementary differentials
  corresponding to rooted trees.
  \begin{enumerate}[label=(\roman*)]
  \item The integrator map $ \Bseries{[]} (f) = f $ is the identity,
    so it is trivially seen to be $F$-functionally equivariant with
    respect to all maps $F$.\label{i:identity}

  \item Consider the integrator map
    $ \Bseries{[[]]} (f) = f ^\prime f $. Applying this to the
    augmented vector field gives
    \begin{align*}
      \Bseries{[[]]} (g) 
      &= \bigl(  f ^\prime f , F ^\prime f ^\prime f + F ^{ \prime \prime } ( f , f ) \bigr) \\
      &= \Bigl( \Bseries{[[]]} (f) , F ^\prime \Bseries{[[]]} (f) + F ^{ \prime \prime } \bigl( \Bseries{[]} (f) , \Bseries{[]} (f) \bigr)  \Bigr) .
    \end{align*}
    If $F$ is affine, then $ F ^{ \prime \prime } = 0 $, so the last
    term vanishes and $ \Bseries{[[]]} $ is $F$-functionally
    equivariant. However, if $ F ^{ \prime \prime } \neq 0 $, then
    this generally does not hold. Thus, $ \Bseries{[[]]} $ is affine
    functionally equivariant, as guaranteed by
    \cref{p:ae_implies_afe}, but not quadratic functionally
    equivariant.

  \item Consider the integrator map
    $ \Bseries{[[[]]]} (f) = f ^\prime f ^\prime f $, whose
    application to the augmented vector field is
    \begin{align*}
      \Bseries{[[[]]]} (g)
      &= \bigl( f ^\prime f ^\prime f , F ^\prime f ^\prime f ^\prime f + F ^{ \prime \prime } ( f ^\prime f, f ) \bigr) \\
      &= \Bigl( \Bseries{[[[]]]} (f) , F ^\prime \Bseries{[[[]]]}(f) + F ^{ \prime \prime } \bigl( \Bseries{[[]]} (f), \Bseries{[]}(f) \bigr) \Bigr) .
    \end{align*}
    As in the previous example, $ \Bseries{[[[]]]} $ is affine
    functionally equivariant, since $ F ^{ \prime \prime } = 0 $ for
    affine $F$, but not quadratic functionally equivariant. \label{i:I_example}

  \item Consider the integrator map
    $ \Bseries{[[][]]} (f) = f ^{ \prime \prime } ( f, f ) $, whose
    application to the augmented vector field is
    \begin{align*}
      \Bseries{[[][]]} (g)
      &= \bigl( f ^{ \prime \prime } ( f, f ) , F ^\prime f ^{ \prime \prime } ( f, f ) + 2 F ^{ \prime \prime } ( f ^\prime f , f ) + F ^{ \prime \prime \prime } ( f, f, f ) \bigr) \\
      &= \Bigl( \Bseries{[[][]]} (f) , F ^\prime \Bseries{[[][]]} (f) + 2 F ^{ \prime \prime } \bigl( \Bseries{[[]]} (f) , \Bseries{[]}(f) \bigr) + F ^{ \prime \prime \prime } \bigl( \Bseries{[]}(f), \Bseries{[]}(f), \Bseries{[]}(f) \bigr)  \Bigr) 
    \end{align*}
    As in the last two examples, $ \Bseries{[[][]]} $ is affine
    functionally equivariant, since the $ F ^{ \prime \prime } $ and
    $ F ^{ \prime \prime \prime } $ terms vanish, but not quadratic
    functionally equivariant.

  \item Finally, consider the integrator map
    $ \phi (f) = f ^\prime f ^\prime f - \tfrac{1}{2} f ^{ \prime
      \prime } ( f , f ) $, i.e.,
    $ \phi = \Bseries{[[[]]]} - \tfrac{1}{2} \Bseries{[[][]]}
    $. Combining the calculations in the previous two examples, we get
    \begin{align*}
      \phi (g)
      &= \Bigl( f ^\prime f ^\prime f - \tfrac{1}{2} f ^{ \prime \prime } ( f, f ) , F ^\prime \bigl( f ^\prime f ^\prime f - \tfrac{1}{2} f ^{ \prime \prime } ( f, f ) \bigr) - \tfrac{1}{2} F ^{ \prime \prime \prime } ( f, f, f ) \Bigr)\\
      &= \Bigl( \phi (f) , F ^\prime \phi (f) - \tfrac{1}{2} F ^{ \prime \prime \prime } \bigl( \Bseries{[]}(f), \Bseries{[]}(f), \Bseries{[]}(f) \bigr)  \Bigr) ,
    \end{align*}
    where subtraction causes the $ F ^{ \prime \prime } $ terms to
    cancel. Thus, $ \phi $ is quadratic functionally equivariant,
    since $ F ^{ \prime \prime \prime } = 0 $ for quadratic $F$, but
    not cubic functionally equivariant. Regarding the general
    impossibility of cubic functional equivariance for B-series
    methods (i.e., B-series other than the exact flow), see
    \citet[Corollary~2.10(c)]{McSt2022}, which uses results of
    \citet{ChMu2007} and \citet{IsQuTs2007} on cubic invariant
    preservation.\label{i:qfe_example}
  \end{enumerate}
  We remark that the modified vector field for the implicit midpoint
  method is given by a B-series
  \begin{equation*}
    \Bseries{[]} +  \Bigl( \frac{ 1 }{ 12 } \Bseries{[[[]]]} - \frac{ 1 }{ 24 } \Bseries{[[][]]} \Bigr) + \cdots .
  \end{equation*}
  From \ref{i:identity} and \ref{i:qfe_example}, we see that the terms
  at each order are quadratic functionally equivariant integrator
  maps, corresponding to the fact that the implicit midpoint method is
  quadratic functionally equivariant. This is an example of a more
  general link between functional equivariance of integrators and
  their modified vector fields, which will be explored in
  \cref{s:mvf}.
\end{example}

\subsection{Closure under differentiation and observables involving
  variations}

We are often interested in the evolution of observables of the
\emph{variational equation}
\begin{equation}
  \label{e:variational}
  \dot{y} = f (y) , \qquad \dot{ \eta } = f ^\prime (y) \eta ,
\end{equation}
where $ \eta \in Y $ is called a \emph{variation} of $y$. For example,
the canonical symplectic two-form is a quadratic observable depending
on two variations of $y$, and it is invariant whenever $f$ is a
canonical Hamiltonian vector field. In order to describe the numerical
evolution of this and other observables depending on variations, we
develop the notion of \emph{closure under differentiation} for affine
equivariant integrator maps. The idea of closure under differentiation
and its connection with symplecticity, particularly for Runge--Kutta
methods, was pioneered by \citet{BoSc1994}. We adapt the approach of
\citet[Section~2.3]{McSt2022} for affine equivariant integrators.

\begin{definition}
  \label{d:closed_under_differentiation}
  Given $ f \in \mathfrak{X} (Y) $, define
  $ \delta f \in \mathfrak{X} ( Y \times Y ) $ by
  $ \delta f ( y, \eta ) = \bigl( f (y) , f ^\prime (y) \eta \bigr) $,
  corresponding to the variational system \eqref{e:variational}. An
  integrator map $\phi$ is \emph{closed under differentiation} if
  $ \phi ( \delta f ) = \delta \phi (f) $ for all $f$.
\end{definition}

\begin{remark}
  The vector field $ \delta f $ is called the \emph{tangent lift} of
  $f$ by \citet{BoSc1994}; elsewhere, it is called the \emph{complete
    lift} of $f$, cf.~\citet{YaKo1966}.
\end{remark}

Compare the following result with \citep[Theorem~2.12]{McSt2022}.

\begin{theorem}
  \label{t:closed_under_differentiation}
  Affine equivariant integrator maps are closed under differentiation.
\end{theorem}

\begin{proof}
  Given $ f \in \mathfrak{X} (Y) $, consider the system
  \begin{equation*}
    \dot{x} = f (x) , \qquad \dot{y} = f (y) ,
  \end{equation*}
  corresponding to the vector field
  $ f \times f \in \mathfrak{X} ( Y \times Y ) $. Since
  $ f \times f \sim _A f$, where $A$ is either of the projections
  $ ( x , y ) \mapsto x $ or $ ( x, y ) \mapsto y $, affine
  equivariance of $\phi$ implies that
  $ \phi (f \times f ) \sim _A \phi (f) $ and thus
  $ \phi ( f \times f ) = \phi (f) \times \phi ( f ) $. Now, taking $ F ( x, y ) = ( x - y ) / \epsilon $ for $ \epsilon > 0 $ gives the augmented system
  \begin{equation}
    \label{e:fxf_augmented}
    \dot{x} = f (x) , \qquad \dot{y} = f (y) , \qquad \dot{z} = \frac{ f (x) - f (y) }{ \epsilon } .
  \end{equation}
  Since $F$ is affine, \cref{p:ae_implies_afe} says that applying
  $\phi$ to this augmented system gives
  \begin{equation}
    \label{e:fxf_augmented_modified}
    \dot{ \widetilde{x} } = \widetilde{f} ( \widetilde{x}) , \qquad \dot{ \widetilde{y} } = \widetilde{f} ( \widetilde{y} ) , \qquad \dot{ \widetilde{z} } = \frac{ \widetilde{f} ( \widetilde{x} ) - \widetilde{f} ( \widetilde{y} ) }{ \epsilon } ,
  \end{equation}
  which is the augmented system of
  $ \phi ( f \times f ) = \phi (f) \times \phi (f) $. Now, letting
  $ x = y + \epsilon \eta $ in \eqref{e:fxf_augmented} and taking
  $ \epsilon \rightarrow 0 $, the $z$-component converges to
  $ f ^\prime (y) \eta $. Similarly, letting
  $ \widetilde{x} = \widetilde{y} + \epsilon \widetilde{ \eta } $ in
  \eqref{e:fxf_augmented_modified} and taking
  $ \epsilon \rightarrow 0 $, the $\widetilde{ z } $-component
  converges to
  $ \widetilde{ f } ^\prime ( \widetilde{y} ) \widetilde{ \eta }
  $. Since $\phi$ is smooth and maps \eqref{e:fxf_augmented} to
  \eqref{e:fxf_augmented_modified} for all $\epsilon$, we conclude
  that $ \phi ( \delta f ) = \delta \phi (f) $, which completes the
  proof.
\end{proof}

We immediately obtain the following corollary for observables of the
variational equation; compare \citep[Corollary~2.13]{McSt2022}.

\begin{corollary}
  \label{c:variational_observables}
  Let $ f \in \mathfrak{X} (Y) $ and
  $ F \colon Y \times Y \rightarrow Z $, and suppose $\phi$ is affine
  equivariant and $F$-functionally equivariant.  If
  $ g \in \mathfrak{X} ( Y \times Y \times Z ) $ is the augmented
  vector field of $ \delta f $,
  \begin{equation*}
    g ( y, \eta , z ) = \Bigl( f (y) , f ^\prime (y) \eta , F ^\prime ( y, \eta ) \bigl( f (y) , f ^\prime (y) \eta \bigr) \Bigr) ,
  \end{equation*}
  then $ \widetilde{g} $ is the augmented vector field of
  $ \delta \widetilde{f} $,
  \begin{equation*}
    \widetilde{g} ( \widetilde{y}, \widetilde{\eta} , \widetilde{z} ) = \Bigl( \widetilde{f} ( \widetilde{y} ), \widetilde{f} ^\prime ( \widetilde{y} ) \widetilde{ \eta } , F ^\prime ( \widetilde{y} , \widetilde{\eta} ) \bigl( \widetilde{f} ( \widetilde{y} ) , \widetilde{f} ^\prime ( \widetilde{y} ) \widetilde{ \eta } \bigr)  \Bigr) .
  \end{equation*}
  That is,
  $ \phi \bigl( ( \delta f , F ^\prime \delta f ) \bigr) = \bigl(
  \delta \phi (f) , F ^\prime \delta \phi (f) \bigr) $.
\end{corollary}

\begin{proof}
  We have
  $ \phi \bigl( ( \delta f , F ^\prime \delta f ) \bigr) = \bigl( \phi
  ( \delta f ) , F ^\prime \phi ( \delta f ) \bigr) = \bigl( \delta
  \phi (f), F ^\prime \delta \phi (f) \bigr) $, where the first
  equality holds by $F$-functional equivariance and the second
  equality holds by closure under differentiation.
\end{proof}

We can easily extend \cref{c:variational_observables} to observables
depending on two or more variations. For instance, if $\xi$ and $\eta$
are each variations of $y$, then
$ ( y , \xi, \eta ) \in Y \times Y \times Y $ satisfies
\begin{equation}
  \label{e:two_variations}
  \dot{y} = f (y) , \qquad \dot{ \xi } = f ^\prime (y) \xi, \qquad \dot{ \eta } = f ^\prime ( y ) \eta .
\end{equation}
This is $A$-related to $\delta f$, where $A$ is either of the
projections $ ( y , \xi, \eta ) \mapsto ( y , \xi ) $ or
$ ( y , \xi, \eta ) \mapsto ( y , \eta ) $, so affine equivariance
of $\phi$ implies that we have
\begin{equation}
  \label{e:two_variations_modified}
  \dot{ \widetilde{y} } = \widetilde{f} ( \widetilde{y} ) , \qquad \dot{ \widetilde{ \xi } } = \widetilde{f} ^\prime ( \widetilde{y} ) \widetilde{ \xi } , \qquad \dot{ \widetilde{ \eta } } = \widetilde{f} ^\prime ( \widetilde{y} ) \widetilde{ \eta } .
\end{equation}
If $\phi$ is $F$-functionally equivariant for some
$ F \colon Y \times Y \times Y \rightarrow Z $, we may then conclude
that it maps the augmented vector field of \eqref{e:two_variations} to
that of \eqref{e:two_variations_modified}.

A particularly important instance of this, which generalizes the
result that quadratic invariant preserving B-series are
symplectic, is worked out in the following example.

\begin{example}
  \label{e:symplectic}
  If $ F ( y, \xi, \eta ) = \omega ( \xi, \eta ) $, where
  $ \omega \colon Y \times Y \rightarrow Z $ is a continuous bilinear
  map, then we augment \eqref{e:two_variations} by the equation
  $ \dot{z} = ( L _f \omega ) _y ( \xi, \eta ) $. Here,
  $ ( L _f \omega ) _y $ is the Lie derivative of $\omega$ along $f$
  at $y$. Hence, the augmented vector field is
  \begin{equation*}
    g ( y, \xi , \eta , z ) = \bigl( f (y) , f ^\prime (y) \xi, f ^\prime (y) \eta , ( L _f \omega ) _y ( \xi, \eta ) \bigr) .
  \end{equation*}
  If $\phi$ is quadratic functionally equivariant, then it follows
  that $ \widetilde{g} $ is the augmented vector field of
  \eqref{e:two_variations_modified},
  \begin{equation*}
    \widetilde{g} ( \widetilde{y} , \widetilde{ \xi } , \widetilde{ \eta } , \widetilde{z} ) = \bigl( \widetilde{f} ( \widetilde{y} ) , \widetilde{f} ^\prime ( \widetilde{y} ) \widetilde{ \xi } , \widetilde{f} ^\prime ( \widetilde{y} ) \widetilde{ \eta } , ( L _{ \widetilde{f} } \omega ) _{ \widetilde{y} } ( \widetilde{ \xi }, \widetilde{ \eta } ) \bigr) .
  \end{equation*}
  In particular, if $ L _f \omega = 0 $, then
  \cref{t:fe_iff_invariant} implies that
  $ L _{ \widetilde{f} } \omega = 0 $ as well. As a special case, if
  $ ( Y , \omega ) $ is a symplectic vector space and $f$ is a
  symplectic vector field, then $ \widetilde{f} $ is also symplectic.
\end{example}

\subsection{Quadratic functionally equivariant B-series}
\label{s:qfe}

\citet{McMoMuVe2016} proved that affine equivariant integrator maps
are precisely those that can be represented by a B-series.  Let $T$
denote the set of rooted trees. As in \cref{e:b-series_fe}, we
identify each $ \tau \in T $ with the integrator map taking $f$ to its
corresponding elementary differential: $ \Bseries{[]} (f) = f $,
$ \Bseries{[[]]} (f) = f ^\prime f $,
$ \Bseries{[[[]]]} (f) = f ^\prime f ^\prime f $,
$ \Bseries{[[][]]} (f) = f ^{ \prime \prime } ( f, f ) $, etc. We can
thus express any affine equivariant integrator map as a B-series
\begin{equation*}
  \phi = \sum _{ \tau \in T } \frac{ b ( \tau ) }{ \sigma (\tau) } \tau ,
\end{equation*}
where $ \sigma (\tau) $ is the symmetry coefficient of $\tau$. This
section will assume that the reader is familiar with B-series, and we
refer to \citet{HaLuWa2006} and \citet{Butcher2021} for a
comprehensive treatment.

\citet[Theorem~IX.9.3]{HaLuWa2006} prove that the truncated modified
vector field of a B-series integrator is symplectic, and thus also
quadratic invariant preserving, if and only if
\begin{equation}
  \label{e:qfe_condition}
  b ( u \circ v ) + b ( v \circ u ) = 0 , \quad \forall u, v \in T ,
\end{equation} 
where $ \circ $ is the Butcher product on rooted
trees. \hyperref[i:qfe_example]{\cref*{e:b-series_fe}\ref*{i:qfe_example}},
which has $ b ( \Bseries{[[[]]]} ) = 1 $ and
$ b ( \Bseries{[[][]]} ) = - 1 $, can be seen to satisfy this
condition condition, since
$ \Bseries{[]} \circ \Bseries{[[]]} = \Bseries{[[[]]]} $ and
$ \Bseries{[[]]} \circ \Bseries{[]} = \Bseries{[[][]]} $. The proof
given in \citep[Theorem~IX.9.3]{HaLuWa2006} makes use of the
symplecticity criterion of \citet{CaSa1994} for B-series integrators
in terms of their coefficients $ a ( \tau ) $, along with a recursion
formula relating these to the $ b ( \tau ) $ coefficients of the
modified vector field. We remark that preservation of quadratic
invariants by Runge--Kutta methods was first characterized by
\citet{Cooper1987}.

Here, we give a new, direct proof of the criterion
\eqref{e:qfe_condition} for quadratic functional equivariance of
integrator maps. Our task is simplified substantially by the fact that
quadratic functional equivariance describes the evolution of
observables for \emph{arbitrary} vector fields $f$, not just
Hamiltonian vector fields or those preserving a particular invariant.

\begin{theorem}
  \label{t:qfe_b-series}
  A B-series integrator map is quadratic functionally equivariant if
  and only if its coefficients satisfy
  $ b ( u \circ v ) + b ( v \circ u ) = 0 $ for all $ u , v \in T $.
\end{theorem}

\begin{proof}
  We write each rooted tree as
  $ \tau = [ \tau _1 , \ldots, \tau _m ] $, denoting that the root of
  $\tau$ has $m$ children, which are the roots of subtrees
  $ \tau _1 , \ldots, \tau _m $.  Let $ F \colon Y \rightarrow Z $ be
  quadratic, meaning that $ F ^{ \prime \prime \prime } = 0
  $. Applying $\tau$ to $ g = ( f, F ^\prime f ) $ and using the
  Leibniz rule to differentiate the product $ F ^\prime f $ gives
  \begin{equation*}
    \tau (g) = \biggl( \tau (f) , F ^\prime \tau (f) + \sum _{ i = 1 } ^m F ^{ \prime \prime } \bigl( [ \tau _1 , \ldots , \widehat{ \tau } _i , \ldots , \tau _m ] (f), \tau _i (f) \bigr) \biggr) ,
  \end{equation*}
  where $ \widehat{ \tau } _i $ denotes that the subtree $ \tau _i $
  is omitted.  Letting
  $ u = [ \tau _1 , \ldots, \widehat{ \tau } _i , \ldots, \tau _m ] $
  and $ v = \tau _i $, we see that each term in the sum above can be
  rewritten as $ F ^{ \prime \prime } \bigl( u (f) , v ( f) \bigr) $
  with $ u \circ v = \tau $. However, this term can appear more than
  once if $ v = \tau _i $ for multiple values of $i$. Recall that the
  symmetry coefficient $ \sigma (\tau) $ is defined recursively by
  \begin{equation*}
    \sigma ( \tau ) = \sigma ( \tau _1 ) \cdots \sigma ( \tau _m ) \mu _1 ! \cdots \mu _k ! , 
  \end{equation*}
  where $ \mu _1, \ldots , \mu _k $ count the number of occurrences of
  each unique tree in the list $ \tau _1 , \ldots, \tau _m $,
  $ k \leq m $. If $ \mu _j $ is the number of times $ \tau _i $
  appears in $\tau$, then
  \begin{equation*}
    \sigma (u) = \sigma ( \tau _1 ) \cdots \widehat{ \sigma ( \tau _i ) } \cdots \sigma ( \tau _m ) \mu _1 ! \cdots ( \mu _j -1 ) ! \cdots \mu _k !,
  \end{equation*}
  and since $ v = \tau _i $, it follows that
  \begin{equation*}
    \mu _j = \frac{ \sigma (u \circ v) }{ \sigma (u) \sigma (v) } .
  \end{equation*}
  This is precisely the number of times
  $ F ^{ \prime \prime } \bigl( u (f) , v ( f) \bigr) $ appears in the
  sum, which we now rewrite as
  \begin{equation*}
\sum _{ i = 1 } ^m F ^{ \prime \prime } \bigl( [ \tau _1 , \ldots , \widehat{ \tau } _i , \ldots , \tau _m ] (f), \tau _i (f) \bigr) = \sum _{ u \circ v = \tau } \frac{ \sigma ( u \circ v ) }{ \sigma (u) \sigma (v) } F ^{ \prime \prime } \bigl( u (f) , v (f) \bigr) .
  \end{equation*}
  Thus, summing the terms of the B-series and rewriting
  $ \sum _{ \tau \in T } \sum _{ u \circ v = \tau } $ as a sum over
  $ u, v \in T $, we get
  \begin{equation*}
    \phi (g) = \biggl( \phi (f) , F ^\prime \phi  (f) + \sum _{ u, v \in T } \frac{ b ( u \circ v ) }{ \sigma (u) \sigma (v) } F ^{ \prime \prime } \bigl( u (f) , v (f) \bigr) \biggr) .
  \end{equation*}
  Now, $\phi$ is quadratic functionally equivariant if and only if the
  extra terms in this sum cancel for all $f$ and $F$. We have
  $ F ^{ \prime \prime } \bigl( u (f) , v (f) \bigr) = F ^{ \prime
    \prime } \bigl( v (f) , u (f) \bigr) $, by symmetry of the
  Hessian, but no other relations among the terms in general. (See
  below for further discussion of this claim.)  Hence, these extra
  terms cancel generically if and only if
  $ b ( u \circ v ) + b ( v \circ u ) = 0 $ for all $u, v \in T $.
\end{proof}

The ``only if'' conclusion depends on the fact that, for all
$ u, v \in T $, one may construct $f$ and $F$ such that
$ F ^{ \prime \prime } \bigl( u (f) , v (f) \bigr) = F ^{ \prime
  \prime } \bigl( v (f) , u (f) \bigr) $ are the only nonvanishing
Hessian terms at some point. When $F$ corresponds to the canonical
symplectic form, \citet[Lemma~5.1]{CaSa1994} construct a Hamiltonian
vector field $f$ that has this property. (See also
\citet[Theorem~VI.7.4]{HaLuWa2006}, which gives a similar construction
that works for both B-series and P-series.) We give a self-contained
proof in \cref{s:necessity}, where the construction of $f$ and $F$ is
simplified by the fact that we are not constrained to the symplectic
setting.

\section{Modified vector fields of functionally equivariant integrators}
\label{s:mvf}

\subsection{Modified vector fields and relatedness}

We now consider the case where we are given a numerical integrator
$\Phi$ rather than an integrator map. As in
\citet[Chapter~IX]{HaLuWa2006}, we suppose that $ \Phi _{ h f } $ may
be expanded as a formal power series in $h$,
\begin{equation*}
  \Phi _{ h f } = \mathrm{id} + h f + h ^2 {d} _2 + h ^3 {d} _3 + \cdots .
\end{equation*}
We then seek a modified vector field $ \widetilde{f} $, also expressed
as a formal power series,
\begin{equation*}
  \widetilde{f} = f + h f _2 + h ^2 f _3 + \cdots ,
\end{equation*}
such that $ \Phi _{ h f } = \exp h \widetilde{f} $, in the sense that
the power series match term-by-term. Matching these terms yields a
recurrence for $ f _j $ in terms of the given $ {d} _j $
\citep[Equation~IX.1.4]{HaLuWa2006}.

We begin by extending the notion of $\chi$-relatedness to modified
vector fields, then prove a general result linking this to properties
of the integrator $\Phi$. This result will be instrumental in
characterizing the modified vector fields of affine equivariant and
functionally equivariant integrators.

\begin{definition}
  Let $ \widetilde{f} = f + h f _2 + h ^2 f _3 + \cdots $ and
  $ \widetilde{g} = g + h g _2 + h ^2 g _3 + \cdots $, where
  $ f , f _2, f _3 , \ldots \in \mathfrak{X} (Y) $ and
  $ g , g _2 , g _3 , \ldots \in \mathfrak{X} (U) $. Given
  $ \chi \colon Y \rightarrow U $, define
  $ \widetilde{f} \sim _\chi \widetilde{g} $ to mean that
  $ f \sim _\chi g $ and $ f _j \sim _\chi g _j $ for all
  $j = 2 , 3, \ldots ,$ i.e., $ \widetilde{f} $ and $ \widetilde{g} $
  are term-by-term $\chi$-related.
\end{definition}

\begin{theorem}
  \label{t:chi-related}
  Given an integrator $\Phi$, let $ \widetilde{f} $ and
  $ \widetilde{g} $ be the modified vector fields of $f$ and $g$,
  respectively. If
  $ \chi \circ \Phi _{ h f } = \Phi _{ h g } \circ \chi $ for all
  sufficiently small $h$, then
  $ \widetilde{f} \sim _\chi \widetilde{ g } $. Furthermore, if
  $ \chi \circ \Phi _{ h f } $ and $ \Phi _{ h g } \circ \chi $ are
  both real analytic in $h$ at $ h = 0 $, then the converse holds.
\end{theorem}

\begin{proof}
  If $ \widetilde{f} $ and $ \widetilde{g} $ are actual vector fields,
  then
  $ \chi \circ \exp h \widetilde{f} = \exp h \widetilde{g} \circ \chi
  $ for all $h$ if and only if
  $ \widetilde{ f } \sim _\chi \widetilde{ g } $. This is a standard
  result on vector fields and flows,
  cf.~\citet[Proposition~4.2.4]{AbMaRa1988}. To extend this to the
  case where $ \widetilde{f} $ and $ \widetilde{g} $ are formal power
  series in $h$, we use an induction argument on the truncations,
  which are actual vector fields.

  First, observe that
  \begin{alignat*}{2}
    \chi \circ \Phi _{ h f } (y) &= \chi \circ \bigl[ \mathrm{id} + h f + \mathcal{O} ( h ^2 ) \bigr] (y) &&= \chi (y) + h \chi ^\prime (y)  f (y) + \mathcal{O} ( h ^2 ) ,\\
    \Phi _{ h g } \circ \chi (y) &= \bigl[ \mathrm{id} + h g + \mathcal{O} ( h ^2 ) \bigr] \circ \chi (y) &&= \chi (y) + h g \bigl( \chi (y) \bigr) + \mathcal{O} ( h ^2 ) ,
  \end{alignat*}
  where we have linearized $\chi$ about $y$ on the first line.
  Matching terms implies $ f \sim _\chi g $, which establishes the
  base case. For the induction step, suppose that
  $ f \sim _\chi g , \ldots , f _{ j -1 } \sim _\chi g _{ j -1 }
  $. Then
  \begin{align}
    \chi \circ \Phi _{ h f } (y)
    &= \chi \circ \bigl[ \exp ( h f + \cdots + h ^j f _j ) + \mathcal{O} ( h ^{ j + 1 } ) \bigr] (y) \notag \\
    &= \chi \circ \bigl[ \exp ( h f + \cdots + h ^{ j -1 } f _{ j -1 } ) + h ^j f _j + \mathcal{O} ( h ^{ j + 1 } ) \bigr] (y) \notag \\
    &= \chi \circ \exp ( h f + \cdots + h ^{ j -1 } f _{ j -1 } ) (y) + h ^j \chi ^\prime (y) f _j (y) + \mathcal{O} ( h ^{ j + 1 } ) . \label{e:ftilde}
  \end{align}
  Here, the second line uses the fact that all higher-order
  exponential terms involving $ h ^j f _j $ are
  $ \mathcal{O} ( h ^{ j + 1 } ) $, and the third line linearizes
  $\chi$ about
  $ \exp ( h f + \cdots + h ^{ j -1 } f _{ j -1 } ) (y) $. Similarly,
  \begin{align}
    \Phi _{ h g } \circ \chi (y)
    &= \bigl[ \exp ( h g + \cdots + h ^j g _j ) + \mathcal{O} ( h ^{ j + 1 } ) \bigr] \circ \chi (y) \notag \\
    &= \bigl[ \exp ( h g + \cdots + h ^{j-1} g _{j-1} ) + h ^j g _j + \mathcal{O} ( h ^{ j + 1 } ) \bigr] \circ \chi (y) \notag \\
    &= \exp ( h g + \cdots + h ^{j-1} g _{j-1} ) \circ \chi (y) + h ^j g _j \bigl( \chi (y) \bigr) + \mathcal{O} ( h ^{ j + 1 } ) . \label{e:gtilde}
  \end{align}
  By the inductive assumption,
  $ h f + \cdots + h ^{ j -1 } f _{ j -1 } \sim _\chi h g + \cdots + h
  ^{ j -1 } g _{ j -1 } $, i.e., the truncations are
  $\chi$-related. Therefore, applying
  \citep[Proposition~4.2.4]{AbMaRa1988}, we have
  \begin{equation*}
    \chi \circ \exp ( h f + \cdots + h ^{ j -1 } f _{ j -1 } ) = \exp ( h g + \cdots + h ^{j-1} g _{j-1} ) \circ \chi.
  \end{equation*}
  Hence, if \eqref{e:ftilde} and \eqref{e:gtilde} are equal, then we
  may cancel the first term of each to conclude
  $ f _j \sim _\chi g _j $.
  
  Conversely, the analyticity assumption allows us to conclude the
  equality of $ \chi \circ \Phi _{ h f } (y) $ and
  $ \Phi _{ h g } \circ \chi (y) $ for sufficiently small $h$ from the
  equality of their power series.
\end{proof}

\begin{remark}
  This can be seen as a generalization of Theorems~IX.5.1 and IX.5.2
  in \citet{HaLuWa2006}, which cover the case where $\chi$ is a
  parametrization of a manifold.
\end{remark}

\subsection{Affine equivariant integrators}

The terms $ {d} _j $ and $ f _j $ in the power series for
$ \Phi _{ h f } $ and $ \widetilde{f} $, respectively, depend on $f$
but not on $h$. Hence, they may be seen as arising from integrator
maps $ \delta _j (f) = {d} _j $ and $ \phi _j (f) = f _j $. Moreover,
each of these is homogeneous of degree $j$, meaning that
$ \delta _j ( h f ) = h ^j {d} _j $ and
$ \phi _j ( h f ) = h ^j f _j $.

We now discuss the relationship between affine equivariance of the
integrator $\Phi$ and that of the integrator maps $ \delta _j $ and
$ \phi _j $.

\begin{definition}
  An integrator $\Phi$ is \emph{affine equivariant} if $ f \sim _A g $
  implies $ A \circ \Phi _f = \Phi _g \circ A $ whenever $A$ is an
  affine map.
\end{definition}

\begin{corollary}
  \label{c:ae}
  If the integrator $\Phi$ is affine equivariant, then so are the
  integrator maps $\phi _j $. The converse is true if
  $ \Phi _{ h f } $ is real analytic in $h$ at $ h = 0 $ for all $f$.
\end{corollary}

\begin{proof}
  Apply \cref{t:chi-related}, where $ \chi = A $ is any affine map.
\end{proof}

An analogous result is true for the $ \delta _j $. One way to see this
would be to apply the main result of \citet{McMoMuVe2016} to conclude
that $ \widetilde{f} $ is a B-series in $f$, then use the fact that
the exponential of a B-series is also a B-series. However, the
following self-contained, tree-free proof uses only the basic
properties of affine equivariance.

\begin{proposition}
  \label{p:ae_delta}
  If the integrator $\Phi$ is affine equivariant, then so are the
  integrator maps $ \delta _j $. The converse is true if
  $ \Phi _{ h f } $ is real analytic in $h$ at $ h = 0 $ for all $f$.
\end{proposition}

\begin{proof}
  Let $ f \sim _A g $ for some affine map $A$. Since $A$ is affine, we
  have
  $ A \circ \Phi _{ h f } - A = A ^\prime \circ ( \Phi _{ h f } -
  \mathrm{id} ) $, and therefore
  \begin{alignat*}{5}
    A \circ \Phi _{ h f } &= A &&+ h (A ^\prime \circ f) &&+ h ^2 (A ^\prime \circ {d} _2) &&+ h ^3 (A ^\prime \circ {d} _3) &&+ \cdots .\\
    \intertext{Next, writing $ \delta _j ( g ) = e _j $, we have}
    \Phi _{ h g } \circ A &= A &&+ h ( g \circ A ) &&+ h ^2 ( e _2 \circ A ) &&+ h ^3 ( e _3 \circ A ) &&+ \cdots .
  \end{alignat*}
  Thus, if $ A \circ \Phi _{ h f } = \Phi _{ h g } \circ A $, then the
  power series agree term-by-term, which means that
  $ {d} _j \sim _A e _j $. Conversely, assuming analyticity, equality
  of the power series implies equality of the maps.
\end{proof}

\begin{remark}
  The forward direction is essentially a version of the ``transfer
  argument'' in \citep[Proposition~6.2]{McMoMuVe2016}. Since $ \Phi $
  is an affine equivariant \emph{integrator}, $ \Phi - \mathrm{id} $
  is an affine equivariant \emph{integrator map}. Thus, the terms
  $ \delta _j $ in the Taylor series of $ \Phi - \mathrm{id} $ at $0$
  are also affine equivariant integrator maps.
\end{remark}

\subsection{Functionally equivariant integrators}

We now consider the relationship between functional equivariance of
integrators, in the sense of \citet{McSt2022}, and that of the
integrator maps $ \phi _j $ constituting the terms of the modified
vector field. We first recall the definition from \citep{McSt2022}
stated in the introduction.

\begin{definition}
  Given a G\^ateaux differentiable map $ F \colon Y \rightarrow Z $, a
  numerical integrator $\Phi$ is \emph{$F$-functionally equivariant}
  if
  $ ( \mathrm{id}, F ) \circ \Phi _f = \Phi _g \circ ( \mathrm{id}, F
  ) $ for all $ f \in \mathfrak{X} (Y) $, where
  $ g \in \mathfrak{X} (Y \times Z) $ is the augmented vector field of
  $f$. Given a class of maps $\mathcal{F}$, the integrator is
  \emph{$\mathcal{F}$-functionally equivariant} if this holds for all
  $ F \in \mathcal{F} ( Y, Z ) $ and all Banach spaces $Y$ and $Z$.
\end{definition}

\begin{corollary}
  \label{c:fe}
  If the integrator $\Phi$ is $F$-functionally equivariant, then so
  are the integrator maps~$ \phi _j $. The converse is true if $F$ is
  real analytic and if $ \Phi _{ h f } $ is real analytic in $h$ at
  $ h = 0 $ for all $f$.
\end{corollary}

\begin{proof}
  Apply \cref{t:chi-related}, where $ \chi = ( \mathrm{id}, F ) $.
\end{proof}

\begin{remark}
  Unlike many of the results in \cref{s:integrator_maps}, this
  corollary does not require the additional assumption of affine
  equivariance. However, if we \emph{do} have affine equivariance,
  then combining \cref{c:ae,c:fe} links the results of
  \cref{s:integrator_maps} for affine equivariant integrator maps to
  the corresponding results of \citet{McSt2022} for affine equivariant
  integrators.
\end{remark}

\begin{remark}
  Except in the case where $F$ is affine, we generally do \emph{not}
  have a version of \cref{p:ae_delta}, in either direction, for
  $F$-functional equivariance of the integrator maps $ \delta _j
  $. For example, the implicit midpoint method
  $ \Phi = \mathrm{id} + \Bseries{[]} + \frac{1}{2} \Bseries{[[]]} +
  \cdots $ is a quadratic functionally equivariant integrator, but
  $ \delta _2 = \frac{1}{2} \Bseries{[[]]} $ is not a quadratic
  functionally equivariant integrator map, as shown in
  \hyperref[i:I_example]{\cref*{e:b-series_fe}\ref*{i:I_example}}. On
  the other hand, Euler's method $ \Phi = \mathrm{id} + \Bseries{[]} $
  has $ \delta _j = 0 $ for all $ j = 2, 3, \ldots ,$ and trivially
  $ 0 \sim _{ (\mathrm{id}, F ) } 0 $ for any $F$ whatsoever, but
  Euler's method is not quadratic functionally equivariant.
\end{remark}

\section{Generalization to additive and partitioned methods}
\label{s:additive_partitioned}

The integrator maps and integrators discussed in the preceding
sections include Runge--Kutta and B-series methods, but not additive
methods (such as additive Runge--Kutta and NB-series methods
\citep{ArMuSa1997} and splitting/composition methods \citep{McQu2002})
or partitioned methods (such as partitioned Runge--Kutta methods and
P-series methods \citep{Hairer1980}). In this section, we briefly
discuss the extension of the foregoing theory to these two classes of
methods. For each class, we modify the notion of integrator map and
functional equivariance, similarly to how this was done for
integrators in \citet[Section~5]{McSt2022}.

\subsection{Additive methods}

An additive method is applied to a vector field
$ f \in \mathfrak{X} (Y) $ after it has been decomposed into a sum
$ f = f ^{ [1] } + \cdots + f ^{ [N] } $, and different decompositions
of the same $f$ may yield different numerical trajectories.

\begin{definition}
  An \emph{additive integrator map} is a collection of smooth maps
  \begin{equation*}
    \phi _Y \colon \underbrace{\mathfrak{X} (Y) \times \cdots \times \mathfrak{X}  (Y)} _N  \rightarrow \mathfrak{X} (Y)
  \end{equation*}
  for each Banach space $Y$, where $ N \in \mathbb{N} $ is the same
  for all $Y$. We denote the application of $\phi$ to
  $ f = f ^{ [1] } + \cdots + f ^{ [N] } \in \mathfrak{X} (Y) $ by
  $ \widetilde{f} = \phi ( f ^{ [1] } , \ldots, f ^{ [N] } ) $, where
  it is understood that $ \widetilde{f} $ depends on the decomposition
  and not just on $f$ itself.
\end{definition}

The following definitions extend the notions of affine equivariance
and functional equivariance to additive integrator maps.

\begin{definition}
  An additive integrator map $\phi$ is \emph{$N$-affine equivariant}
  if, for all affine maps $A$, we have
  $ \phi ( f ^{ [1] } , \ldots, f ^{ [N] } ) \sim _A \phi ( g ^{ [1] }
  , \ldots, g ^{ [N] } ) $ whenever
  $ f ^{ [\nu] } \sim _A g ^{ [\nu] } $ for all
  $ \nu = 1 , \ldots, N $.
\end{definition}

\begin{definition}
  Given a G\^ateaux differentiable map $ F \colon Y \rightarrow Z $
  and $ f ^{ [1] } , \ldots, f ^{ [N] } \in \mathfrak{X} (Y) $, let
  $ g ^{ [\nu] } ( y, z ) = \bigl( f ^{[\nu]} (y) , F ^\prime (y) f
  ^{[\nu]} (y) \bigr) \in \mathfrak{X} (Y \times Z) $ be the augmented
  vector field of $ f ^{ [\nu] } $ for $ \nu = 1 , \ldots, N $. An
  additive integrator map $\phi$ is \emph{$F$-functionally
    equivariant} if
  $ \phi ( f ^{ [1] } , \ldots, f ^{ [N] } ) \sim _{ (\mathrm{id}, F )
  } \phi ( g ^{ [1] } , \ldots, g ^{ [N] } ) $ for all
  $ f ^{ [1] } , \ldots, f ^{ [N] } \in \mathfrak{X} (Y) $, and
  \emph{$\mathcal{F}$-functionally equivariant} if this holds for all
  $ F \in \mathcal{F} ( Y, Z ) $ and all Banach spaces $Y$ and $Z$.
\end{definition}

We next prove a version of \cref{p:gtilde_affine}, which strengthens
the notion of functional equivariance for $N$-affine equivariant
methods.

\begin{proposition}
  If $\phi$ is $N$-affine equivariant, then
  $ \widetilde{g} ( \widetilde{y} , \widetilde{z} ) = \widetilde{g}
  \bigl( \widetilde{y} , F ( \widetilde{y} ) \bigr) $ for all
  $ ( \widetilde{y}, \widetilde{z} ) \in Y \times Z $. Consequently,
  the $F$-functional equivariance condition \eqref{e:fe_mvf} holds if
  and only if
  $ \widetilde{g} ( \widetilde{y} , \widetilde{z} ) = \bigl(
  \widetilde{f} ( \widetilde{y} ) , F ^\prime ( \widetilde{y} )
  \widetilde{f} ( \widetilde{y} ) \bigr) $.
\end{proposition}

\begin{proof}
  The proof is essentially the same as that of
  \cref{p:gtilde_affine}. Consider the affine map
  $ A ( y, z ) = ( y, z + c ) $, where $ c \in Z $ is a
  constant. Since each $ g ^{ [\nu] } $ depends only on $y$, we have
  $ g ^{ [\nu] } \sim _A g ^{ [\nu] } $ for $ \nu = 1, \ldots, N
  $. Thus, $N$-affine equivariance implies
  $ \widetilde{g} \sim _A \widetilde{g} $, i.e.,
  $ \widetilde{g} ( \widetilde{y}, \widetilde{z} ) = \widetilde{g} (
  \widetilde{y}, \widetilde{z} + c ) $, so taking
  $ c = F ( \widetilde{y} ) - \widetilde{z} $ completes the proof.
\end{proof}

Compare the following with \citep[Proposition~5.3]{McSt2022}.

\begin{proposition}
  Every $N$-affine equivariant integrator map is affine functionally
  equivariant.
\end{proposition}

\begin{proof}
  As in the proof of \cref{p:ae_implies_afe}, if $F$ is affine, then
  so is $ ( \mathrm{id}, F ) $. Since
  $ f ^{ [\nu] } \sim _{ (\mathrm{id}, F) } g ^{ [\nu] } $ for
  $ \nu = 1 , \ldots, N $, if $\phi$ is $N$-affine equivariant, then
  $ \phi ( f ^{ [1] } , \ldots, f ^{ [N] } ) \sim _{ (\mathrm{id}, F )
  } \phi ( g ^{ [1] } , \ldots, g ^{ [N] } ) $.
\end{proof}

Before proving an additive version of \cref{t:fe_iff_invariant}, we
note an important distinction between ordinary and additive integrator
maps. Affine equivariant integrator maps preserve affine invariants,
since $ f \sim _A 0 $ implies $ \phi (f) \sim _A \phi (0) = 0 $. In
contrast, it is possible for an additive integrator map to be
$N$-affine equivariant but \emph{not} preserve affine invariants of
$f$, since $ f \sim _A 0 $ does not necessarily imply
$ f ^{ [\nu] } \sim _A 0 $ for all $ \nu = 1 , \ldots, N $, unless $A$
is also an invariant of each individual $ f ^{ [\nu] } $. This is
illustrated in the following examples.

\begin{example}
  \label{e:nb_affine}
  We consider affine invariant preservation of some simple NB-series
  with $ N = 2 $. These can be represented in terms of trees with
  black and white vertices, cf.~\citet*{ArMuSa1997}.
  \begin{enumerate}[label=(\roman*)]
  \item The integrator maps
    $ \Bseries{[]} (f ^{ [1] }, f ^{ [2] } ) = f ^{ [1] } $ and
    $ \Bseries{[o]} (f ^{ [1] }, f ^{ [2] } ) = f ^{ [2] } $ do not
    necessarily preserve affine invariants of $f$, since we may have
    $ A ^\prime f = 0 $ but $ A ^\prime f ^{ [\nu] } \neq 0 $ for
    $ \nu = 1, 2 $. However, the integrator map
    $ ( \Bseries{[]} + \Bseries{[o]} ) (f ^{ [1] }, f ^{ [2] } ) = f $
    clearly does preserve affine invariants of $f$.

  \item Similarly,
    $ \Bseries{[[]]} (f ^{ [1] }, f ^{ [2] } ) = f ^{ [1] \prime } f
    ^{ [1] } $ and
    $ \Bseries{[o[]]} (f ^{ [1] }, f ^{ [2] } ) = f ^{ [2] \prime } f
    ^{ [1] } $ do not necessarily preserve affine invariants of $f$.
    However,
    $ ( \Bseries{[[]]} + \Bseries{[o[]]} ) (f ^{ [1] }, f ^{ [2] } ) =
    f ^\prime f ^{ [1] } $ does, since $ A ^\prime f = 0 $ implies
    that
    $ A ^\prime f ^\prime f ^{[1]}= ( A ^\prime f ) ^\prime f ^{[1]} =
    0 $.
  \end{enumerate}
  In general, the condition for NB-series to preserve affine
  invariants of $f$, for arbitrary decompositions, is that trees must
  have the same coefficients if they differ only in the color of their
  roots. Indeed, if $ [ \tau _1 , \ldots, \tau _m ] ^{ [\nu] } $
  denotes the $N$-colored tree whose root has the color $\nu$, and
  whose children are the roots of the subtrees
  $ \tau _1, \ldots, \tau _m $, then these having equal coefficients
  allows us to collect the terms
  \begin{align*} 
    A ^\prime \sum _{ \nu = 1 } ^N [ \tau _1 , \ldots , \tau _m ] ^{ [\nu] } ( f ^{ [1] }, \ldots, f ^{ [N] } )
    &= A ^\prime \sum _{ \nu = 1 } ^N ( f ^{ [\nu] } ) ^{ (m) } \Bigl(  \tau _1 ( f ^{ [1] }, \ldots, f ^{ [N] } ) , \ldots, \tau _m ( f ^{ [1] }, \ldots, f ^{ [N] } ) \Bigr) \\
    &= ( A ^\prime f ) ^{ (m) } \Bigl(  \tau _1 ( f ^{ [1] }, \ldots, f ^{ [N] } ) , \ldots, \tau _m ( f ^{ [1] }, \ldots, f ^{ [N] } ) \Bigr) ,
  \end{align*} 
  which vanishes if $ A ^\prime f = 0 $.
\end{example}

Thus, the $ ( \Leftarrow ) $ direction of \cref{t:fe_iff_invariant}
does not hold for $N$-affine equivariant integrators, even when
$\mathcal{F}$ is the class of affine maps, but it does hold if we
require the the additional condition of affine invariant
preservation. Compare the following with
\citep[Theorem~5.7]{McSt2022}.

\begin{theorem}
  \label{t:N_fe_iff_invariant}
  Let $\mathcal{F}$ satisfy \cref{a:affine}. An $N$-affine equivariant
  integrator map $\phi$ is $\mathcal{F}$-invariant preserving if and
  only if it is $\mathcal{F}$-functionally equivariant and affine
  invariant preserving.
\end{theorem}

\begin{proof}
  $ ( \Rightarrow ) $ Suppose $\phi$ is $\mathcal{F}$-invariant
  preserving. The proof of $\mathcal{F}$-functional equivariance is
  essentially the same as in \cref{t:fe_iff_invariant}. The only
  notable modification is that, at the step where $ A ( y, z ) = y $
  is the linear projection onto $Y$, we use
  $ g ^{ [\nu] } \sim _A f ^{ [\nu] } $ for $ \nu = 1 , \ldots , N $
  to conclude that $ \widetilde{g} \sim _A \widetilde{f} $ by
  $N$-affine equivariance. Furthermore, affine invariant preservation
  follows from $\mathcal{F}$-invariant preservation, since
  \cref{a:affine} implies that $\mathcal{F}$ includes all affine maps.

  $ ( \Leftarrow ) $ Conversely, suppose that $\phi$ is
  $\mathcal{F}$-functionally equivariant and affine invariant
  preserving. Just as in the proof of \cref{t:fe_iff_invariant}, if
  $ F \in \mathcal{F} ( Y, Z ) $ is an invariant of
  $ f \in \mathfrak{X} (Y) $, then
  $ g ( y, z ) = \bigl( f (y) , 0 \bigr) $, and
  $\mathcal{F}$-functional equivariance implies
  $ \widetilde{g} ( \widetilde{y}, \widetilde{z} ) = \bigl(
  \widetilde{f} ( \widetilde{y} ) , F ^\prime ( \widetilde{y} )
  \widetilde{f} ( \widetilde{y} ) \bigr) $. Finally, the linear
  projection $ B ( y, z ) = z $ is an affine invariant of $g$ and thus
  of $ \widetilde{g} $, so
  $ F ^\prime ( \widetilde{y} ) \widetilde{f} ( \widetilde{y} ) = 0 $.
\end{proof}

We next give an integrator map version of
\citep[Corollary~5.9]{McSt2022}, which says that a consistent
splitting method cannot preserve affine invariants unless it equals
the exact flow. Modified vector fields of splitting methods contain
only the terms $ f ^{ [1] } , \ldots, f ^{ [N] } $ and their iterated
Jacobi--Lie brackets, as a consequence of the
Baker--Campbell--Hausdorff formula
\citep[Section~IX.4]{HaLuWa2006}. Since Jacobi--Lie brackets of
$\chi$-related vector fields are $\chi$-related for \emph{any} $\chi$
whatsoever \citep[Proposition~4.2.25]{AbMaRa1988}, it follows that
modified vector fields of splitting methods are $N$-affine equivariant
and $F$-functionally equivariant with respect to \emph{all} maps $F$.
To prove that any consistent integrator map with this property must
agree with the exact flow, we first establish the following
consequence of consistency.

\begin{lemma}
  \label{l:constant_splitting}
  Suppose an $N$-affine equivariant integrator map $\phi$ satisfies
  the consistency condition
  $ \phi ( h f ^{ [1] }, \ldots, h f ^{ [N] } ) = h f + o (h) $. If
  $ f ^{ [\nu] } $ is constant for all $\nu = 1 , \ldots, N $, then
  $ \phi ( f ^{ [1] }, \ldots , f ^{ [N] } ) = f $.
\end{lemma}

\begin{proof}
  Let $ A (y) = y + c $ for some constant $ c \in Y $. If each
  $ f ^{ [\nu] } $ is constant, then
  $ f ^{ [\nu] } \sim _A f ^{ [\nu] } $, so $N$-affine equivariance
  implies $ \widetilde{f} \sim _A \widetilde{f} $, i.e.,
  $ \widetilde{f} $ is also constant. On the other hand, letting
  $ B (y) = h y $, we have $ f ^{ [\nu] } \sim _B h f ^{ [\nu] } $, so
  applying $N$-affine equivariance again, we have
  $ \widetilde{f} \sim _B \widetilde{ h f} $, i.e.,
  $ h \widetilde{f} (\widetilde{y}) = \widetilde{ h f } ( h
  \widetilde{y} ) $. Since both $ h \widetilde{f} $ and
  $ \widetilde{ h f } $ are constant, we get
  $ h \widetilde{f} = \widetilde{ h f } = h f + o (h) $ by the
  consistency condition. Thus, $ \widetilde{f} = f + o (1) $, but
  since this does not depend on $h$ at all, we conclude that
  $ \widetilde{f} = f $.
\end{proof}

\begin{theorem}
  Consider consistent, $N$-affine equivariant integrator maps that are
  $F$-functionally equivariant for all maps $F$ (e.g., those arising
  from splitting methods). The unique such integrator map preserving
  affine invariants is
  $ \phi ( f ^{ [1] } , \ldots, f ^{ [N] } ) = f $, i.e., the exact
  flow.
\end{theorem}

\begin{proof}
  By \cref{t:N_fe_iff_invariant}, any such $\phi$ preserving affine
  invariants must preserve \emph{all} invariants. Given
  $ f \in \mathfrak{X} (Y) $, consider the vector field
  $ ( f, 1 ) \in \mathfrak{X} ( Y \times \mathbb{R} ) $. This augments
  $ \dot{y} = f (y) $ by the equation $ \dot{t} = 1 $, where $t$ may
  be seen as time. Thus, $ F ( y, t ) = \exp ( - t f ) (y) $ is an
  invariant of $ ( f, 1 ) $.

  Now, suppose we split $ ( f , 1 ) $ into
  $ ( f ^{ [\nu] } , c ^{ [\nu] } ) $, where
  $ c ^{ [\nu] } \in \mathbb{R} $ are constants summing to $1$. By
  \cref{l:constant_splitting} and $N$-affine equivariance with respect
  to $ ( y, t ) \mapsto t $, we have
  $ \widetilde{ ( f , 1 ) } = ( \widetilde{f} , 1 ) $. Since $\phi$
  preserves the invariant $F$, we conclude that
  $ \exp ( - t f ) ( \widetilde{y} ) $ is an invariant of
  $ ( \widetilde{f} , 1 ) $, and thus $ \widetilde{f} = f $.
\end{proof}

\begin{remark}
  Without the consistency hypothesis, which allows us to use
  \cref{l:constant_splitting}, $ \dot{ t } = 1 $ does not necessarily
  imply $ \dot{ \widetilde{ t } } = 1 $. This allows for the
  possibility $ \widetilde{f} = c f $ with $ c \neq 1 $, giving a time
  reparametrization of the exact flow.
\end{remark}

Finally, we note that \cref{t:chi-related} holds virtually unchanged
for the relationship of additive integrators to their modified vector
fields, where we need only replace $ \Phi _{ h f } $ by
$ \Phi _{ h f ^{ [1] }, \ldots, h f ^{ [N] } } $ and $ \Phi _{ h g } $
by $ \Phi _{ h g ^{ [1] }, \ldots, h g ^{ [N] } } $. Thus, we
immediately get additive-integrator versions of \cref{c:ae} for
$N$-affine equivariance and \cref{c:fe} for $F$-functional
equivariance, in the sense of \citep[Section~5.1]{McSt2022}.

\subsection{Partitioned methods}

A partitioned method is based on a partitioning
$ Y = Y ^{ [1] } \oplus \cdots \oplus Y ^{ [N] } $. These are similar
to additive methods, except the decomposition
$ f = f ^{ [1] } + \cdots + f ^{ [N] } $ is uniquely determined by the
partitioning of $Y$.

\begin{definition}
  A \emph{partitioned integrator map} is a collection of smooth maps
  \begin{equation*}
    \phi _{ Y ^{ [1] } \oplus \cdots \oplus Y ^{ [N] } } \colon \mathfrak{X}  ( Y ) \rightarrow \mathfrak{X}  (Y) ,
  \end{equation*}
  for each partitioned Banach space
  $ Y = \bigoplus _{ \nu = 1 } ^N Y ^{ [\nu] } $, where
  $ N \in \mathbb{N} $ is fixed. For a given partitioning of $Y$, we
  simply write $ \widetilde{f} = \phi (f) $ for
  $ f \in \mathfrak{X} (Y) $.
\end{definition}

For partitioned methods, rather than considering all affine maps
between Banach spaces, we consider particular affine maps that respect
the partitioning, cf.~\citep[Definition~5.10]{McSt2022}.

\begin{definition}
  Given partitioned spaces
  $ Y = \bigoplus _{ \nu = 1 } ^N Y ^{ [\nu] } $ and
  $ U = \bigoplus _{ \nu = 1 } ^N U ^{ [\nu] } $, a map
  $ A \colon Y \rightarrow U $ is \emph{P-affine} if it decomposes as
  $ A = \bigoplus _{ \nu = 1 } ^N A ^{ [\nu] } $, where each
  $ A ^{ [\nu] } \colon Y ^{ [\nu] } \rightarrow U ^{ [\nu] } $ is
  affine. A partitioned integrator map $\phi$ is \emph{P-affine
    equivariant} if $ f \sim _A g $ implies
  $ \phi (f) \sim _A \phi (g) $ for all P-affine maps $A$.
\end{definition}

\begin{remark}
  When $A$ is P-affine, $ f \sim _A g $ is equivalent to
  $ f ^{ [\nu] } \sim _{ A ^{ [\nu] } } g ^{ [\nu] } $ for all
  $ \nu = 1 , \ldots, N $.
\end{remark}

In particular, let $ A \colon Y \rightarrow \mathbb{R} $ be an affine
functional, and partition $ \mathbb{R} ^{ [\mu] } = \mathbb{R} $ and
$ \mathbb{R} ^{ [\nu] } = \{ 0 \} $ for $ \nu \neq \mu $. Then $A$ is
P-affine if and only if $ A = A ^{ [\mu] } $, i.e., $A$ depends only
on $ Y ^{ [\mu] } $ \citep[Example~5.11]{McSt2022}.

The following integrator-map version of
\citep[Proposition~5.12]{McSt2022} shows how $N$-affine equivariant
integrator maps (e.g., NB-series) give rise to P-affine equivariant
integrator maps (e.g., P-series).

\begin{proposition}
  If an additive integrator map $\psi$ is $N$-affine equivariant, then
  the partitioned integrator map
  $ \phi (f) = \psi ( f ^{ [1] } , \ldots, f ^{ [N] } ) $ is P-affine
  equivariant.
\end{proposition}

\begin{proof}
  This follows immediately from the definitions, since P-affine maps
  are affine.
\end{proof}

The definition of $F$- and $\mathcal{F}$-functional equivariance is
the same as in \cref{d:fe}, where given partitions
$ Y = \bigoplus _{ \nu = 1 } ^N Y ^{ [\nu] } $ and
$ Z = \bigoplus _{ \nu = 1 } ^N Z ^{ [\nu] } $, we apply $\phi$ to the
augmented vector field by partitioning the product
$ Y \times Z = \bigoplus _{ \nu = 1 } ^N ( Y ^{ [\nu] } \times Z ^{
  [\nu] } ) $. To prove a partitioned version of
\cref{t:fe_iff_invariant}, we must introduce a P-affine version of
\cref{a:affine}; cf.~\citep[Assumption~5.14]{McSt2022}. Important
examples of $\mathcal{F}$ satisfying these assumptions are P-affine
maps, all affine maps, quadratic maps that are at most bilinear with
respect to the partition, and all quadratic maps,
cf.~\citep[Examples~5.16--5.19]{McSt2022}.

\begin{assumption}
  \label{a:p-affine}
  Assume that:
  \begin{itemize}
  \item $\mathcal{F} (Y,Y) $ contains the identity map for all
    $Y = \bigoplus _{ \nu = 1 } ^N Y ^{[\nu]}$;
  \item $ \mathcal{F} ( Y, Z ) $ is a vector space for all
    $Y = \bigoplus _{ \nu = 1 } ^N Y ^{[\nu]}$ and
    $Z=\bigoplus _{ \nu = 1 } ^N Z ^{[\nu]}$;
  \item $ \mathcal{F} $ is invariant under composition with P-affine
    maps, in the following sense: If $ A \colon Y \rightarrow U $ and
    $ B \colon V \rightarrow Z $ are P-affine and
    $ F \in \mathcal{F} ( U, V ) $, then
    $ B \circ F \circ A \in \mathcal{F} ( Y, Z ) $, for all
    $Y = \bigoplus _{ \nu = 1 } ^N Y ^{[\nu]}$,
    $Z = \bigoplus _{ \nu = 1 } ^N Z ^{[\nu]}$,
    $U = \bigoplus _{ \nu = 1 } ^N U ^{[\nu]}$, and
    $V = \bigoplus _{ \nu = 1 } ^N V ^{[\nu]}$.
  \end{itemize}
\end{assumption}

Compare the following to \citep[Theorem~5.15]{McSt2022}.

\begin{theorem}
  Let $\mathcal{F}$ satisfy \cref{a:p-affine}. A P-affine equivariant
  integrator map $\phi$ is $\mathcal{F}$-invariant preserving if and
  only if it is $\mathcal{F}$-functionally equivariant.
\end{theorem}

\begin{proof}
  The proof is essentially identical to that of
  \cref{t:fe_iff_invariant}. For the $ ( \Rightarrow ) $ direction, we
  use the fact that the linear projection
  $ A \colon Y \times Z \rightarrow Y $ is P-affine, since it
  decomposes into the projections
  $ A ^{ [\nu] } \colon Y ^{ [\nu] } \times Z ^{ [\nu] } \rightarrow Y
  ^{ [\nu] } $. For the $ ( \Leftarrow ) $ direction, we similarly use
  the fact that $ B \colon Y \times Z \rightarrow Z $ is P-affine,
  since it decomposes into
  $ B ^{ [\nu] } \colon Y ^{ [\nu] } \times Z ^{ [\nu] } \rightarrow Z
  ^{ [\nu] } $. At the final step, we have $ \widetilde{ 0 } = 0 $,
  since the affine map from the trivial Banach space (with trivial
  partitioning) to any point of $Z$ is P-affine.
\end{proof}

\cref{t:chi-related} holds unchanged for the relationship of
partitioned integrators to their modified vector fields. Thus, we
immediately get partitioned-integrator versions of \cref{c:ae} for
P-affine equivariance and \cref{c:fe} for $F$-functional equivariance,
in the sense of \citep[Section~5.2]{McSt2022}.

\subsection{Closure under differentiation and symplecticity}

Finally, we generalize \cref{t:closed_under_differentiation} to
$N$-affine and P-affine equivariant methods, allowing the functional
equivariance results to be extended to observables depending on
variations. \cref{d:closed_under_differentiation} of closure under
differentiation is formally unchanged for partitioned integrator maps;
for additive integrator maps, we modify it in the obvious way, as
follows.

\begin{definition}
  An additive integrator map $\phi$ is \emph{closed under
    differentiation} if
  \begin{equation*}
    \phi ( \delta f ^{ [1] } , \ldots, \delta f ^{ [N] } ) = \delta
    \phi ( f ^{ [1] } , \ldots, f ^{ [N] } ).
  \end{equation*} 
\end{definition}

Compare the following to \citep[Theorem~5.20]{McSt2022}.

\begin{theorem}
  $N$-affine and P-affine integrator maps are closed under
  differentiation.
\end{theorem}

\begin{proof}
  The proof is essentially the same as that of
  \cref{t:closed_under_differentiation}. The only modification needed
  is to specify the additive decomposition or partitioning of the
  augmented system \eqref{e:fxf_augmented} to which $\phi$ is applied
  in order to obtain \eqref{e:fxf_augmented_modified}. If $\phi$ is an
  additive integrator then we decompose \eqref{e:fxf_augmented} into
  \begin{equation*}
    f (x) = \sum _{ \nu = 1 } ^N f ^{ [\nu] } (x) ,\qquad f (y) = \sum _{ \nu = 1 } ^N f ^{ [\nu] } (y) , \qquad \frac{ f (x) - f (y) }{ \epsilon } = \sum _{ \nu = 1 } ^N \frac{ f ^{ [\nu] } (x) - f ^{ [\nu] } (y) }{ \epsilon } .
  \end{equation*}
  If $\phi$ is P-affine equivariant, we partition
  $ Y \times Y \times Y = \bigoplus _{ \nu = 1 } ^N ( Y ^{ [\nu] }
  \times Y ^{ [\nu] } \times Y ^{ [\nu] } ) $.
\end{proof}

When $ \omega \colon Y \times Y \rightarrow Z $ is a continuous
bilinear map on $Y$, it follows that \cref{e:symplectic} extends
\emph{mutatis mutandis} to $N$-affine and P-affine equivariant
integrator maps---and in particular, those preserving quadratic
invariants are symplectic.

For NB-series, a similar argument to \cref{t:qfe_b-series} gives the
quadratic functional equivariance condition
$ b ( u \circ v ) + b ( v \circ u ) = 0 $ for all $N$-colored trees
$u$ and $v$. Together with the affine invariant preservation condition
that $ b (\tau) $ is independent of the color of the root
(\cref{e:nb_affine}), we recover a modified-vector-field version of
\citet[Theorem~3]{ArMuSa1997}, which states that these maps must
therefore correspond to ordinary symplectic B-series.

On the other hand, if $\omega$ is at most bilinear with respect to a
partition $ Y = \bigoplus _{ \nu = 1 } ^N Y ^{ [\nu] } $ (i.e., the
$ Y ^{ [\nu] } \times Y ^{ [\nu] } $ blocks are trivial), then
$ F ^{ \prime \prime } \bigl( u (f) , v (f) \bigr) = 0 $ when $u$ and
$v$ have the same colored root. Hence, the condition
$ b ( u \circ v ) + b ( v \circ u ) = 0 $ need only hold for trees
with different-colored roots. In particular, when $\omega$ is the
canonical symplectic form on $ Y = E \oplus E ^\ast $, applying this
with $ N = 2 $ recovers \citet[Theorem~IX.10.4]{HaLuWa2006} on modified
vector fields of symplectic P-series.

\appendix

\section{Necessity proof for quadratic functionally equivariant B-series}
\label{s:necessity}

In this appendix, we prove that for all $ u, v \in T $, it is possible
to construct a vector field $f$ and quadratic $F$ such that
$ F ^{ \prime \prime } \bigl( u (f) , v (f) \bigr) = F ^{ \prime
  \prime } \bigl( v (f) , u (f) \bigr) $ is the only nonvanishing
Hessian term at some point (i.e., $ F ^{ \prime \prime } $ vanishes on
all other pairs of trees). This completes the proof of
\cref{t:qfe_b-series} by establishing the necessity of the condition
$ b ( u \circ v ) + b ( v \circ u ) = 0 $ for quadratic functional
equivariance of a B-series integrator map.

We begin with a vector field construction for an individual tree
$\tau$, which we subsequently apply to $u$ and $v$ to prove the claim
above. Let $ \lvert \tau \rvert $ denote the order of $\tau$, i.e.,
its number of vertices.

\begin{lemma}
  \label{l:tauf}
  Given $ \tau \in T $, there exists a vector field $f$ on
  $ \mathbb{R} ^{ \lvert \tau \rvert } $ such that
  \begin{equation*}
    \theta (f) _{\lvert \tau \rvert} (0) =
    \begin{cases}
      \sigma (\tau) ,& \text{if } \theta = \tau ,\\
      0, & \text{otherwise}.
    \end{cases}
  \end{equation*}
\end{lemma}

\begin{proof}
  Label the vertices of $\tau$ by $ 1 , \ldots, \lvert \tau \rvert $,
  where $ i = \lvert \tau \rvert $ is the root. For each vertex $i$ with
  children $ j _1 , \ldots , j _k $, define the $i$th component of $f$
  at
  $ y = ( y _1 , \ldots , y _{ \lvert \tau \rvert } ) \in \mathbb{R}
  ^{ \lvert \tau \rvert } $ to be
  \begin{equation*}
    f _i (y) = y _{ j _1 } \cdots y _{ j _k }.
  \end{equation*}
  In the case $ k = 0 $ (i.e., vertex $i$ is a leaf), we use the
  convention that the empty product is $1$. We claim that this $f$ has
  the desired property.

  The proof of this claim is by induction on the height of $\tau$. The
  base case $ \tau = \Bseries{[]} $ is trivial. For the induction
  step, write $ \tau = [ \tau _1 , \ldots, \tau _m ] $, and suppose
  without loss of generality that the children of the root are
  labeled $ 1 ,\ldots , m $ accordingly. By definition,
  \begin{equation*}
    \tau (f) = f ^{ (m) } \bigl( \tau _1 (f) , \ldots, \tau _m (f) \bigr) .
  \end{equation*}
  Using the inductive assumption, we have
  \begin{equation}
    \label{e:tauj}
    \tau _j (f) _i (0) =
    \begin{cases}
      \sigma ( \tau _j ) ,& \text{if } \tau _i = \tau _j ,\\
      0 ,& \text{otherwise},
    \end{cases}
  \end{equation}
  where $i$ and $j$ range over $ 1, \ldots, m $. Next, since
  $ f _{ \lvert \tau \rvert } (y) = y _1 \cdots y _m $, we have
  \begin{equation}
    \label{e:fm}
    f _{ \lvert \tau \rvert } ^{ (m) } (0) = \sum _{ \pi \in S _m } \mathrm{d} y _{ \pi (1) } \otimes \cdots \otimes \mathrm{d} y _{ \pi (m) } ,
  \end{equation}
  where $ S _m $ is the symmetric group on $m$ elements, so that $\pi$
  is a permutation of $ \{ 1, \ldots, m \} $. Applying this to
  \eqref{e:tauj}, each nonvanishing term evaluates to
  $ \sigma ( \tau _1 ) \cdots \sigma ( \tau _m ) $, and the
  nonvanishing terms correspond to $\pi$ that permute identical trees
  among $ \tau _1 , \ldots, \tau _m $. The number of such permutations
  is precisely $ \mu _1 ! \cdots \mu _k ! $, where the $ \mu _j $
  count the occurrences of each unique tree, as in the proof of
  \cref{t:qfe_b-series}. Therefore,
  \begin{equation*}
    \tau (f) _{ \lvert \tau \rvert } (0) = \sigma ( \tau _1 ) \cdots \sigma ( \tau _m ) \mu _1 ! \cdots \mu _k ! = \sigma (\tau) .
  \end{equation*}
  Finally, suppose $ \theta \neq \tau $, and write
  $ \theta = [ \theta _1 , \ldots, \theta _n ] $. If $ n \neq m $,
  then we have $ f ^{(n)} _{ \lvert \tau \rvert } (0) = 0 $, so
  $ \theta (f) _{ \lvert \tau \rvert } (0) = 0 $. Otherwise,
  $ [ \theta _1, \ldots, \theta _m ] $ is not a permutation of
  $ [ \tau _1 , \ldots, \tau _m ] $, so for all $ \pi \in S _m $,
  there exists some $j$ such that
  $ \theta _j \neq \tau _{ \pi (j) } $. In this case, \eqref{e:tauj}
  implies $ \theta _j (f) _{ \pi (j) } (0) = 0 $, so every term in
  \eqref{e:fm} vanishes when evaluated on
  $ \bigl( \theta _1 (f) (0) , \ldots, \theta _m (f) (0) \bigr) $, and
  again $ \theta (f) _{ \lvert \tau \rvert } (0) = 0 $.
\end{proof}

\begin{lemma}
  Given $ u , v \in T $, there exists a vector field $f$ and quadratic
  functional $F$ on
  $ Y =\mathbb{R} ^{ \lvert u \rvert + \lvert v \rvert } $ such that
  \begin{equation*}
    F ^{ \prime \prime } \bigl( \tau (f), \theta  (f) \bigr) (0) 
    \begin{cases}
      \neq 0, & \text{if $ ( \tau, \theta ) = ( u, v ) $ or $ (\tau, \theta) = ( v, u ) $},\\
      = 0, & \text{otherwise}.
    \end{cases}
  \end{equation*}
\end{lemma}

\begin{proof}
  Similarly to \cref{l:tauf}, label the vertices of $u$ by
  $ 1, \ldots, \lvert u \rvert $, where $ i = \lvert u \rvert $ is the
  root. For each vertex $i$ with children $ j _1 , \ldots, j _k $, let
  \begin{equation*}
    f _i (y) = y _{ j _1 } \cdots y _{ j _k } .
  \end{equation*}
  Repeat this for $v$, labeling its vertices by
  $ \lvert u \rvert + 1 , \ldots, \lvert u \rvert + \lvert v \rvert $,
  where $ i = \lvert u \rvert + \lvert v \rvert $ is the root. Now,
  define the quadratic functional
  \begin{equation*}
    F (y) = y _{ \lvert u \rvert } y _{ \lvert u \rvert + \lvert v \rvert } ,
  \end{equation*}
  so that
  \begin{equation*}
    F ^{ \prime \prime } \bigl( \tau (f) , \theta (f) \bigr) =  \tau (f) _{ \lvert u \rvert } \theta (f) _{ \lvert u \rvert + \lvert v \rvert }  +  \tau (f) _{ \lvert u \rvert + \lvert v \rvert } \theta (f) _{ \lvert u \rvert }.
  \end{equation*}
  If $\tau = u $ and $ \theta = v $, or vice versa, then \cref{l:tauf}
  implies that evaluating this at $ y = 0 $ gives
  $ 2 \sigma ( u ) \sigma (v) $ if $ u = v $ and
  $ \sigma ( u ) \sigma (v) $ if $ u \neq v $. If $\tau$ and $\theta$
  are not $u$ and $v$, then \cref{l:tauf} implies that both terms
  vanish at $ y = 0 $.
\end{proof}

\end{document}